\documentclass{amsart}
\usepackage{amsmath}
\usepackage[english]{babel}
\usepackage{amsmath}
\usepackage{amsfonts}
\usepackage{amssymb}
\usepackage{amsthm}
\usepackage{graphicx} 
\usepackage{url}
\usepackage{cite}
\usepackage[colorlinks]{hyperref}
\hypersetup{
    colorlinks = true,
}
\newtheorem{theorem}{Theorem}[section]
\newtheorem{lemma}[theorem]{Lemma}
\newtheorem{proposition}[theorem]{Proposition}
\newtheorem{corollary}[theorem]{Corollary}

\begin{document} 

\title{On the largest prime divisor of polynomial and related problem}
\author{Thanh Nguyen Cung}
\address{Thanh Nguyen Cung, College of Engineering and Computer Science, VinUniversity, Vietnam}
\email{24thanh.nc@vinuni.edu.vn}

\author{Son Duong Hong} 
\address{Son Duong Hong, Department of Mathematics, University of Warwick, Coventry, United Kingdom }
\email{sonhong.duong101005@gmail.com}

\date{} 
\maketitle

\begin{abstract}
We denote $\mathcal{P}$ = $\{P(x)|$ $P(n) \mid n!$ for infinitely many $n\}$. This article identifies some polynomials that belong to $\mathcal{P}$. Additionally, we also denote $P^+(m)$ as the largest prime factor of $m$. Then, a consequence of this work shows that there are infinitely many $n \in \mathbb{N}$ so that $P^+(f(n)) < n^{\frac{3}{4}+\varepsilon}$ if $f(x)$ is cubic polynomial, $P^+(f(n)) < n$ if $f(x)$ is reducible quartic polynomial and $P^+(f(n)) < n^{\varepsilon}$ if $f(x)$ is Chebyshev polynomial.
\end{abstract} 

\section{Introduction}

It was conjectured in 1857 by Bunyakovsky \cite{bunyakovsky} that an irreducible polynomial $f$ over $\mathbb{Z}$, whose values have no fixed prime divisor will take infinitely many values of $n$ for which $f(n)$ is a prime number. This remains an open problem in mathematics for the case where
$\deg f \ge 2$. Therefore, we will consider a simpler aspect—the largest prime factor of a polynomial, denoted by
$P^+(m)$ as the largest prime factor of $m.$ There are many results on this topic for quadratic and cubic polynomials.  Keates\cite[Theorem 1]{keates} obtain that if $f(x) \in \mathbb{Z}[x]$ is a quadratic or cubic polynomial and distinct roots. Then for all $n$, sufficiently large in absolute value $$P^+(f(n)) > 10^{-7} \log\log n.$$ 
Jori \cite{jori} has shown that $$P^+(n^2+1) > n^{1.279}$$ and 
Hooley \cite{hooley} stated that $$P^+(n^3+2) > n^{1+\frac{1}{30}}$$ holds infinitely many positive integers $n.$ In general, for irreducible polynomials of degree $\deg f \ge 2$ $$P^+(f(n)) > ne^{(\log n)^{\varepsilon}}$$ as proved in \cite{tenenbaum}. Next, the question arises of how small the largest prime factor of a polynomial can be for suitable $n$. In \cite{schinzel}, Schinzel stated that (see \cite{bober}) for $f(x) \in \mathbb{Z}[x]$ of degree $d$, $$P^+(f(n)) < n^{d\theta(d)}$$ for infinitely many positive integers $n$ where  $$\theta(2) = 0.279..., \theta(3) = 0.381..., \theta(4) = 0.559...$$ and for large $d$ $$\theta(d) = 1 - \dfrac{1}{d-2} + O\bigg(\dfrac{1}{d^3}\bigg).$$ 
Schinzel also conjectured that for each $\varepsilon > 0$ and $f(x) \in \mathbb{Z}[x]$, there exist infinitely many $n$ such that $$P^+(f(n)) < n^{\varepsilon.}$$ 
In \cite{bober}, the authors also proved for the case $f(x)$ is a quadratic polynomial, and the authors in \cite{balog} proved for the case $f(x)$ is the product of binomials. Originating from the techniques and proofs in \cite{bober} and \cite{schinzel}, in this paper, we will prove the statement for the case where $f(x)$ is a Chebyshev-type polynomial, as well as improve Schinzel’s result for cubic and reducible quartic polynomials. Additionally, we consider $\mathcal{P}$ = $\{P(x)|$ $P(n) \mid n!$ for infinitely many $n\}$ and identify some polynomials that belong to $\mathcal{P}.$ This means that for each polynomial $f \in \mathcal{P},$ we always have $P^+(f(n)) < n$ for infinite positive integer $n.$    

\section{Preliminaries} 

In this section, we present some of the theorems and lemmas that are fundamental to this article. The first theorem is quite well-known.  
\begin{theorem}
(Schur \cite{schur}). Every nonconstant polynomial $f(x) \in \mathbb{Z}[x]$ has an infinite number of prime divisors. 
\end{theorem}
The second theorem is Legendre's formula giving an expression for the exponent of the largest power of a prime $p$ that divides the factorial $n!.$
\begin{theorem}
(Legendre \cite{legendre}). For any prime number $p$ and any positive integer $n$, let $\nu_p(n)$ be the p-adic valuation of $n$. Then we have $$\nu_p(n!) = \sum_{i=1}^{\infty} \bigg\lfloor \dfrac{n}{p^i} \bigg\rfloor,$$ where $\lfloor x \rfloor$ is the floor function. 
\end{theorem}
Next, we introduce the definition of cyclotomic polynomials and Chebyshev polynomials of the first kind, along with some of their properties.
\begin{itemize}
\item \textbf{Cyclotomic polynomial} \cite{lang}. For each positive integer $n$, the $n^{\text{th}}$ cyclotomic polynomial is defined by the formula 
$$\Phi_{n}(x)=\prod_{\substack{k\le n\\\gcd(k,n)=1}}\left(x-e^{\frac{2k\pi i}{n}}\right)$$
In this article, we only need to consider three important properties of the cyclotomic polynomial, which are:
\begin{itemize}
\item For each positive integer $n$, the $n^{\text{th}}$  cyclotomic polynomial has a degree of $\varphi(n)$.
\item Every cyclotomic polynomial is monic and has integer coefficients.
\item For each positive integer $n$, we always have
    $$x^n-1=\prod_{d \mid n}\Phi_{d}(x).$$ 
\end{itemize} 
\item \textbf{Chebyshev polynomial} \cite{masakazu, rayes}. The Chebyshev polynomials of the first kind are obtained from the recurrence relation: 
\begin{align*}
&T_0(x) = 1 \\
&T_1(x) = x \\
&T_{n+1}(x) = xT_n(x) - T_{n-1}(x)
\end{align*}
In this article, we only need to consider two important properties of the Chebyshev polynomial, which are: 
\begin{itemize}
    \item $T_{mn}(x) = T_m(T_n(x)).$ \\
    \item $2T_n\Big(\dfrac{x}{2}\Big) = \prod_{\substack{d \mid n \\ n/d: \text{odd}}}\psi_{4d}(x)$ where $\psi_n(x) \in \mathbb{Z}[x]$ is the unique polynomial such that $\psi_n(x+x^{-1}) = x^{-\varphi(n)/2}\cdot \Phi_n(x) .$
\end{itemize} 
\end{itemize}
Additionally, we also need the following theorem of Mertens.  
\begin{theorem} (Mertens's third theorem \cite{mertens})
$$\lim_{n \to \infty } \log n\prod\limits_{\substack{p: \text{ prime} \\ p \le n}} \left(1 - \dfrac{1}{p}\right) = e^{-\gamma}.$$ where $\gamma$ is the Euler-Mascheroni constant \cite{euler}. 
\end{theorem} 
From here, we see that a direct consequence of the theorem is as follows.
\begin{corollary} \label{corollary 2.4}
$$\lim_{n \to \infty }\prod\limits_{i=m}^{n} \left(\dfrac{p_i}{p_i-1}\right) = \infty.$$ where $p_i$ is the $i^{\text{th}}$ prime and $m$ is arbitrary positive integer. 
\end{corollary}

\section{Main result} 

Before proceeding to the main result, we observe that for polynomials $P(x)$ of degree at most 4, we will only consider those with positive integer coefficients. Indeed, if the leading coefficient of $P(x)$ is negative, we will instead consider the polynomial $-P(x).$ Moreover, expressing $P(x) = \sum\limits_{i = 0}^{d}a_ix^i,$ we obtain 
\begin{align*}
P(x+y) &= \sum\limits_{i = 0}^{d}a_i(x+y)^i \\
&= \sum\limits_{i = 0}^{d}a_i \sum\limits_{j = 0}^{i}\binom{i}{j}x^{j}y^{i-j}. \\
&= \sum\limits_{i = 0}^{d} \Big(a_d\binom{d}{i}y^{d-i}+Q_i(y)\Big)x^i 
\end{align*}
where $\deg (Q_i(y)) < d-i$ so the coefficient of $x^i$ is a polynomial in $y$ whose leading coefficient is positive (note that $a_d > 0$). Thus, for sufficiently large $y$, the coefficient of $x^i$ will be positive. In this case, we will demonstrate the existence of infinitely many positive integers $n$ such that $P(n+y) \mid n!.$ For this problem, we also observe that it suffices for the greatest proper divisor of $P(n)$ to be smaller than $n$ to achieve the desired satisfaction. Besides this observation, we also need the following two important lemmas as a foundation for the theorems below.

\begin{lemma} \label{lemma 3.1}
Given polynomials $P_1(x),P_2(x),\ldots , P_m(x) \in \mathbb{Z}[x]$ with degrees not exceeding $d$, and polynomial $Q(x) \in \mathbb{Z}[x]$ with positive leading coefficient, degree not less than $d+1.$ Then, there exists $N$ such that for all $n > N,$ we have $$P_1(n) P_2(n)\cdots P_m(n) \mid Q(n)!.$$ 
\end{lemma}

\begin{proof}
Denote $P(x) = P_1(x)P_2(x) \cdots P_m(x)$. Since $P_i(x)$ are polynomials with degrees not exceeding $d$, it follows that $$P_i(x) = O(x^d) \quad \forall i = \overline{1,m}.$$ 
Then, for each prime number $p$ is a divisor of $P(n)$,  we have the following estimate 
\begin{equation}
    \nu_p(P(n)) = \sum\limits_{1\le i \le m}\nu_p(P_i(n)) < O(\log_p(n^d)) < O(\log_2(n))
    \label{eq 3.1}
\end{equation}
On the other hand, since $Q(x)$ is a polynomial with a degree not less than $d+1$, then $Q(x) = O(x^{q})$ where $q = \deg(Q(x))$. From Legendre's formula, we find that
\begin{equation}
\nu_p(Q(n)!) \ge \bigg\lfloor\dfrac{Q(n)}{p}\bigg\rfloor \ge \dfrac{O(n^{d+1})}{O(n^d)} = O(n).
\label{eq 3.2}
\end{equation} 
Hence, from (\ref{eq 3.1}) and (\ref{eq 3.2}) we see that there must exist a number $N$ as desired since $\lim\limits_{n \to \infty} \dfrac{n}{\log_2(n)} = \infty$. 
\end{proof}
Besides the lemma, we also denote $\mathcal{C}$ as the set of polynomials whose coefficients are co-prime. For a more concise presentation, we also consider the mapping 
$$
\begin{aligned}
\tau: &\quad \mathbb{Z}[x] \to \mathcal{C} \\
      &\quad P(x) \mapsto P_1(x)P_2(x) 
\end{aligned}
$$

if there exists a rational number $r$ such that $P(x) = rP_1(x)P_2(x).$  From here, if for $P(x)$ we can construct $\tau(P(x))$ satisfying Lemma \ref{lemma 3.1}, then of course, we deduce that $P(x)$ also satisfies it. 
\begin{lemma} \cite[Lemma 10]{schinzel}
Given a polynomial $f(x) \in \mathbb{Z}[x]$ with positive integer coefficients and degree $d > 2$. Then there exists a polynomial $h(x)$ with integer coefficients and degree $d-1$, with a positive leading coefficient, such that $f(h(x))$ has a factor $g(x) \in \mathbb{Z}[x]$ with degree $d.$ \label{lemma 3.2}
\end{lemma} 

\subsection{Quadratic polynomial} 

\begin{proposition}
Let $P(x) \in \mathbb{Z}[x]$ be a quadratic polynomial. Then, $P(x) \in \mathcal{P}.$ 
\end{proposition}

\begin{proof}
Let $n = P(m)+m$ with $m$ is a positive integer that we choose later. 
Since $$P(P(x)+x) \equiv P(x) \equiv 0 \pmod {P(x)}$$ and $P(P(x)+x)$, $P(x)$ are polynomials with integer coefficients of degrees four and two, respectively, we can see that there must exist a quadratic polynomial $Q(x)$ with integer coefficients and a positive leading coefficient that satisfies $$P(P(x)+x)=P(x)Q(x).$$ 
Now, applying Schur's theorem, we can take a prime $q$ as a sufficiently large prime divisor of $Q(x)$ such that the leading coefficient of $\dfrac{Q(x)}{q}$ is smaller than the leading coefficient of $P(x)$. 
At this point, consider $l$ as a positive integer such that $q \mid Q(l)$. 
Since $Q(x)$ is a polynomial with integer coefficients, we always have that $$Q(m) \equiv Q(l) \equiv 0 \pmod q,$$
for all positive integers $m$ such that $m \equiv l \pmod q$. Thus, combining this with the fact that
$P(x) \text{ and } \dfrac{Q(x)}{q}$ are both quadratic polynomials and the leading coefficient of $\dfrac{Q(x)}{q}$  is smaller than the leading coefficient of $P(x)$, so for any sufficiently large $m$ satisfying $m \equiv l \pmod q$, we must have that $q \mid Q(m)$ and 
$$1<q<\dfrac{Q(m)}{q}<P(m)<P(m)+m.$$
In summary, for infinitely many sufficiently large $m$; $q,\dfrac{Q(m)}{q}$ and $P(m)$ will be distinct positive integers that are all less than $P(P(m)+m)$ and whose product is equal to $P(P(m)+m)$. This implies that $P(P(m)+m) \mid (P(m)+m)!.$ 
\end{proof} 

\subsection{Cubic polynomial}

\begin{theorem}
Let $P(x) \in \mathbb{Z}[x]$ be a cubic polynomial. Then, $P(x) \in \mathcal{P}.$ 
\end{theorem}

\begin{proof} 
Let $P(x) = ax^3+bx^2+cx+d$, where $a,b,c \text{ and }d$ are positive integers. Analyzing the proof of Lemma \ref{lemma 3.2}, consider
\begin{align*}
Q(x) =& (16a^4k^3+8a^3bk^2+16a^3ck-4a^2b^2k-16a^3d+8a^2bc-2ab^3)x^2\\ 
& -(12a^2k^2+4abk+4ac-b^2)x + 2k 
\end{align*}
then there exist two polynomials
\begin{align*}
R(x) = & (16a^4k^3+8a^3bk^2+16a^3ck-4a^2b^2k-16a^3d+8a^2bc-2ab^3)x^3 \\ 
& -(12a^2k^2+4abk+4ac-b^2)x^2 +(3k+2ab)x-1 
\end{align*}  and $S(x)$$ \in \mathbb{Z}[x]$ such that $\tau(P(Q(x)))= R(x)S(x).$ 
\newline 
Set $R(x) = a_rx^3+b_rx^2+c_rx+d_r$ where 
\begin{align*}
&a_r = 16a^4k^3+8a^3bk^2+16a^3ck-4a^2b^2k-16a^3d+8a^2bc-2ab^3 \\ 
&b_r = -(12a^2k^2+4abk+4ac-b^2) \\
&c_r = 3k+2ab \\
&d_r = -1
\end{align*}
and write $S(x) = a_sx^3+b_sx^2+c_sx+d_s.$ Thus, we choose again $$g(x) =Ax^2-Bx+2l$$ with
\begin{align*}
&A =A(l)= 16a_r^4l^3+8a_r^3b_rl^2+16a_r^3c_rl-4a_r^2b_r^2l-16a_r^3d_r+8a_r^2b_rc_r-2a_rb_r^3 \\ & B = B(l) = 12a_r^2l^2+4a_rb_rl+4a_rc_r-b_r^2 
\end{align*}
We also consider the derivative of $A(l)$ given by $$A(l)' = 48a_r^4l^2+16a_r^3b_rl+4a_r^2(4a_rc_r-b_r^2)$$
where the discriminant $\Delta'$ of $A(l)'$ is
$$\Delta '_{A'(l)} = 64a_r^6b_r^2 - 192a_r^6(4a_rc_r-b_r^2) = 64a_r^6(4b_r^2 -12a_rc_r) = 256(b_r^2-3a_rc_r)$$ 
We observe that $b_r^2-3a_rc_c$ is a polynomial in $k$ with the leading coefficient $24a^3b(1-4a^2) < 0$, which is negative. Thus, for sufficiently large $k$, we must have $b_r^2-3a_rc_r < 0$. Consequently, $A(l)' > 0 $ for all $l$ which implies that $A(l)$ has at most one root.
\newline 
On the other hand, we also consider $$h(x) = Cx^2 - Dx+2l$$ with 
\begin{align*} 
& C = C(l) = 16a_s^4l^3+8a_s^3b_sl^2+16a_s^3c_sl-4a_s^2b_s^2l-16a_s^3d_s+8a_s^2b_sc_s-2a_sb_s^3 \\ 
& D = D(l) = 12a_s^2l^2+4a_sb_sl+4a_sc_s-b_s^2 
\end{align*} 
In the next step, we will point out the existence of a positive integer $l$ such that $AC$ is not a perfect square, based on the following lemma.
\begin{lemma} \label{lemma 3.5}
(see \cite{murty}). Let $P(x)$ be a nonconstant polynomial and $P(n)$ is perfect for all sufficiently large $n$. Then there exists a polynomial $Q(x)\in \mathbb{Z}[x]$ such that $P(x) = Q(x)^2.$  
\end{lemma} 
Assume for contradiction that $AC$ is a perfect square for all sufficiently large $l$ or $A(l)C(l)$ is a perfect square for all sufficiently large $l$. By Lemma \ref{lemma 3.5} $$A(l)C(l) = E(l)^2.$$ Since $A(l)$ và $C(l)$ are cubic polynomials, it follows that there must exist linear polynomials with rational coefficients such that $$A(l) = A_1(l)A_2(l)^2.$$ This leads to a contradiction with the choice of $l$. In other words, we see that there must exist an $l$ such that $AC$ is not a perfect square. 
From the choice of $h(x)$ and $g(x)$, by Lemma \ref{lemma 3.2} it follows that there must exist polynomials $R_1(x),R_2(x),S_1(x),S_2(x)$ such that
$$\tau(R(g(x))) = R_1(x)R_2(x)$$ and $$ \tau(S(h(x))) =S_1(x)S_2(x)$$
where $R_1(x),R_2(x),S_1(x),S_2(x)$ are cubic polynomials with positive integer leading coefficients.
\newline 
Next, we will prove that the equation $g(u) = h(v)$ has infinitely many integer solutions $(u,v).$
Indeed, this is equivalent to
\begin{equation}
Au^2-Bu-Cv^2+Dv = 0 
\label{eq 3.3}
\end{equation}  
We know that from $AC$ is not a perfect square, Pell's equation: $$r^2-ACs^2 = 1$$ has infinitely many positive integer solutions $(r,s).$ Therefore, for each pair $(r,s)$, we just need to choose $$u = -BCs^2 - Drs \text{ and } v = -Brs - ADs^2 $$
then $(u,v)$ will be a solution to (\ref{eq 3.3}). From this, there are infinitely many pairs of sufficiently small integers $(u,v)$ such that $g(u) = h(v) > 0$ since $A, B, C, D >0$.
\newline 
Finally, by setting $t = g(u) = h(v),$ we obtain
\begin{align*}
\tau(P(n)) = \tau(P(Q(t))) &= \tau(R(t)S(t)) \\ &= \tau(R(g(u)))\tau( S(h(v))) \\ &= R_1(u)R_2(u)S_1(v)S_2(v)
\end{align*} 
Since $R_1(x),R_2(x),S_1(x),S_2(x)$ are cubic polynomials while $Q(g(x))$ is a quartic polynomial, so for each pair $(u,v)$ when $u \to -\infty$ and $v \to -\infty$, we have $$R_i(u) = O(u^3), S_i(v) = O(v^3)\quad (i = \overline{1,2})$$ and $$Q(t) = O(u^4) = O(v^4)$$  
Therefore, from Lemma \ref{lemma 3.1}, we complete the proof of the desired result.
\end{proof} 
From this, we also see that for each $\varepsilon > 0$, there exist infinitely many $n$ such that $$P^+(P(n)) < n^{\frac{3}{4}+\varepsilon}$$ which is a significant improvement over the result of A. Schinzel \cite{schinzel}($P^+(P(n)) <n^{1.14}$). Moreover, based on Lemma \ref{lemma 3.1}, we also see that $P(x)^m$ also belongs to $\mathcal{P}$ where $m$ is a positive integer.

\subsection{Reducible quartic polynomial} 

\begin{theorem}
Let $P(x) \in \mathbb{Z}[x]$ be a reducible quartic polynomial. Then, $P(x) \in \mathcal{P}.$ 
\end{theorem} 

\begin{proof} 
From the hypothesis, we will consider two cases.

\textbf{Case 1:} $P(x) = (ax^3 + bx^2 + cx + d)(ex+f).$

Similar to Theorem \ref{lemma 3.5}, let $n = Q(m)$ so that  $$\tau(P(Q(m))) = (eQ(m)+f)R(m)S(m)$$ along with $m = g(u) = h(v)$, where $$g(x) = Ax^2-Bx+2k$$ \and $$h(x) = Cx^2-Dx+2k.$$ Thus, we will need to prove.
\begin{align}
(eQ(g(u))+f)R_1(u)R_2(u)S_1(v)S_2(v) \mid Q(g(u))! 
\label{eq 4}
\end{align} 
We denote $p = eP(2k)+f > e$ and consider the equation $g(u) = h(v)$ which means
\begin{equation}
Au^2 - Bu -Cv^2 + Dv = 0. 
\label{eq 3.5}
\end{equation} 
We know that the Pell equation $r^2-ACs^2 = 1$ has infinitely many solutions $(r,s)$ given by 
$$
\begin{cases}
    r_0=1,r_{n+2} = 2r_1r_{n+1}-r_n. \\ 
    s_0 = 0, s_{n+2} = 2r_1s_{n+1}-s_n.
\end{cases}
$$
For each such pair $(r,s)$, we will choose $$(u,v) = (-BCs^2-Drs, -Brs-ADs^2)$$ as a solution pair of (\ref{eq 3.5}) Since $s_0 = 0$, we will show that there exists a subsequence $(s_{n_j})_{j \in \mathbb{N}}$ such that $p $ is a divisor of $s_{n_j}.$ Indeed, consider $p^2+1$ pairs $(s_i, s_{i+1}).$ By the pigeonhole principle, there must exist indices $i_1 < i_2$ such that $$ s_{i_1+2} \equiv s_{i_2+2} \pmod{p} \quad \text{and} \quad s_{i_1+1} \equiv s_{i_2+1} \pmod{p}.$$
From this, based on the recurrence formula of the sequence $(s_n),$ we deduce that $s_{i_1} \equiv s_{i_2} \pmod{p}$ and by induction, we conclude that $s_i \equiv s_{i+j} \pmod{p}$ where $j = i_2-i_1$ for all $i \in \mathbb{N}.$
Thus, from $s_0 = 0,$ we see that there must exist a subsequence $(s_{n_j})$ as desired. 
Therefore, $p$ will be a divisor of $u$ and we have $$eQ(g(u))+f \equiv eQ(g(0)) + f \equiv eP(2k)+f \equiv 0 \pmod{p}$$
From this, $\dfrac{eQ(g(u))+f}{p} < Q(g(u))$ for sufficient large $u$, so applying Lemma \ref{lemma 3.1} we will obtain (\ref{eq 4}).
\newline 

\textbf{Case 2:} $P(x) = (a_1x^2+b_1x+c_1)(a_2x^2+b_2x + c_2) = P_1(x)P_2(x)$. 

Without loss of generality, we can assume that $c_1 = 1$ and $c_2 = max\{a_2,b_2,c_2 \}$ because we can replace $n$ by $(n+m)c_1$ with $m$ large enough so that $P_1((n+m)c_1)=c_1(ac_1(n+m)^2 +b(n+m) + 1) = c_1P_3(n)$ and the constant term of $P_4(n) = P_2((n+m)c)$ is the largest. Therefore, we will consider the problem with 
\begin{align*}
P(x) &= (ax^2+bx+1)(dx^2+ex+f) \\
&= R(x)Q(x)
\end{align*}
where $f = max\{d,e,f\}.$ 
We consider the equation $$k+f+lfQ(k+f) = u+vR(u)$$ where $l,v$ are constants chosen later, and $k,u$ are variables. Then, the equation is equivalent to 
\begin{equation}
alfk^2+(2alf+lfb+1)k+lfQ(f)+f = dvu^2+(ev+1)u+vf
\label{eq 3.6}
\end{equation}
We choose a fixed prime number $p$ that is greater than $\max\{adf,Q(f)\}$ and select $l,v$ such that $lQ(f)+1=v$ with $\nu_p(v) = 1.$ 
We can choose $l$ and $v$ satisfying the stated property because if $(l,v)$ is a solution pair, then $(l+t,v+tQ(f))$ is also a solution pair. 
We will choose $t$ such that $t \equiv \dfrac{p-v}{Q(f)} \pmod{p^2}$ to ensure that $\nu_p(v+tQ(f)) = 1.$ From this choice, we deduce that $alfdv$ will be divisible by $p$ not divisible by $p^2$, implying that it cannot be a perfect square.

Let $A = alf, B = 2alf+lfb+1, C = dv \text{ and } D = ev+1$. Then from Theorem \ref{lemma 3.5}, (\ref{eq 3.6}) will have infinitely many solution pairs $$(k,u) = (BCs^2+Drs,Brs+ADs^2)$$ where $(r,s)$ is a solution pair of the Pell equation $r^2-ACs^2 = 1.$ 
This equation has a sequence of solutions given by the formula 
$$
\begin{cases}
    r_0=1, r_{n+2} = 2r_1r_{n+1}-r_n. \\
    s_0 = 0, s_{n+2} = 2r_1s_{n+1} - s_n.
\end{cases}
$$
Additionally, let $n = k+f+lfQ(k+f) = u+vR(u)$ and using the property that $$f(x+f(x)) \equiv 0 \pmod{f(x)}$$ to have $$Q(n)R(n) = Q_1(k)R_1(u)Q_2(k)R_2(u) $$ where $Q_1(x) = Q(x+f), Q_2(x) = alfx^2+(2alf^2+lfb+1)x+lf(1+bl+af^2)+f \text{ and } R_1(x) = R(x), R_2(x) = x+vR(x)$.
Clearly, we see that the constant terms of $Q_2(x), R_2(x)$ are always greater than their leading coefficients, denoted as $c_q \text{ and } c_r$ respectively. Also, since $s_0 = 0$ we can select a subsequence $(s_{n_j})_{j \in \mathbb{N}}$ such that $c_qc_r$ is a divisor of $s_{n_j}.$ 
From this, since $s$ is always a divisor of $k$ and $u$, we deduce that $c_q$ is a divisor of  $k$ and $c_r$ is a divisor of  $u.$ Thus $$Q_1(k)R_1(u)Q_2(k)R_2(u) = c_qc_rQ_1(k)R_1(u)\dfrac{Q_2(k)}{c_q}\dfrac{R_2(u)}{c_r}$$ with $k,u > N$ for some fixed $N$,we must have $c_qc_r,Q_1(k),R_1(u),\dfrac{Q_2(k)}{c_q},\dfrac{R_2(u)}{c_r}$ as distinct numbers all smaller than $k+f+lfQ(k+f) = u+vR(u).$ Hence $P(n) \mid n!$ for infinitely many $n$ as desired. 
\end{proof}
From the above discussion, we can see that for infinitely many $n$, $P^+{(f(n)) < n}$ for a reducible quartic polynomial $f(x)$. This improves upon Schinzel's result, which stated that $P^+(f(n)) < n^{2.237}$ for infinitely many $n.$  

\subsection{Cyclotomic polynomial}
The question of whether the polynomial $f(x) \in \mathbb{Z}[x]$ is a product of binomials was proven by the authors in \cite{balog}. We will provide a simpler proof; from there, for $f(x),$ we can approach it similarly by using Corollary \ref{corollary 2.4} and Lemma \ref{lemma 3.1}
\begin{proposition}
$P(x) = x^m -1 \in \mathcal{P}$ where $m$ is a positive integer. 
\end{proposition}

\begin{proof}
Applying the Corollary \ref{corollary 2.4}, we find that there exists a positive integer $t$ such that for the first $t$ prime numbers $2 = p_1 < p_2 < \ldots < p_t$, we have 
$$\prod_{j=1}^{t}{ \bigg(\dfrac{p_j}{p_j-1}}\bigg) \gg m.$$
In this case, by choosing $$n = s^{p_1p_2 \cdots p_t}$$ where $s$ is a sufficiently large positive integer, we obtain 
$$ n^{m}-1 = \left(s^{p_1p_2 \cdots p_t}\right)^{m}-1 = \left(s^{m}\right)^{p_1p_2 \cdots p_t}-1 = \prod_{d\mid p_1p_2\cdots p_t}{\Phi_d(s^{m})}.$$
From this, we see that we can express the polynomial $n^m-1$ as a product of polynomials. For each $d \mid p_1p_2 \cdots p_t$, we have $$\varphi(d) \le \varphi(p_1p_2 \cdots p_t)=(p_1-1)(p_2-1) \cdots (p_t-1).$$ so
\begin{align*}
\text{deg}\left(\Phi_d(x^{m})\right)=m\cdot\varphi(d) &\le m\cdot (p_1-1)(p_2-1)\dots (p_t-1)\\&=m \cdot \left(\prod_{j=1}^{t}\left(\dfrac{p_j-1}{p_j}\right)\right)p_1p_2 \cdots p_t \ll p_1p_2 \cdots p_t.
\end{align*}
Therefore, by applying Lemma \ref{lemma 3.1}, we prove the problem.
\end{proof} 
From the above discussion, we also see that if $f(x) = \Phi_m(x)$ is the $m^{\text{th}}$ cyclotomic polynomial, then $f(x)$ also belongs to $\mathcal{P}.$ Moreover, there always exist infinitely many $n$ such that $P^+(f(n)) < n^{\varepsilon}$ where $\varepsilon > 0.$ 

\subsection{Chebyshev polynomial} 
\begin{theorem}
Let $T_m(x)$ be the $m^{\text{th}}$ Chebyshev polynomial of the first kind. Then, $T_m(x) \in \mathcal{P}.$ 
\end{theorem}

\begin{proof}
From Corollary \ref{corollary 2.4}, we can select $ p_1<p_2<\ldots<p_t$ are the first $t$ prime numbers greater than $m$, such that $$\prod_{j=1}^{t}{ \bigg(\dfrac{p_j}{p_j-1}}\bigg) \gg 2\varphi(m).$$
At this point, we choose $n = T_{p_1p_2\cdots p_t}(s)$ with $s$ being a positive integer and set $P = mp_1p_2\cdots p_t$ so that we have 
\begin{align*}
T_m(n) &= T_m(T_{p_1p_2\cdots p_t}(s)) = T_{P}(s) \\ 
& = \dfrac{1}{2} \prod\limits_{d \mid P, P/d:\text{odd}}{\psi_{4d}(2s)} 
\end{align*}
On the other hand, for each $d \mid P$, we have $$\varphi(d) \le \varphi(mp_1p_2 \cdots p_t)=\varphi(m) (p_1-1)(p_2-1) \cdots (p_t-1)$$ so we find that
\begin{align*}
\text{deg}\left(\psi_{4d}(2x)\right)=\dfrac{\varphi(4d)}{2}&\le 2\varphi(m)\cdot (p_1-1)(p_2-1)\dots (p_t-1)\\&=2\varphi(m) \cdot \left(\prod_{j=1}^{t}\left(\dfrac{p_j-1}{p_j}\right)\right)p_1p_2 \cdots p_t \ll p_1p_2 \cdots p_t.
\end{align*}
Thus, with the chosen $n$, we see that $T_m(x)$ can be factored into polynomials with degrees smaller than that of the polynomial $T_{p_1p_2\cdots p_t}(x)$. Therefore, using Lemma \ref{lemma 3.1}, we will obtain the desired result for the problem. 
\end{proof}
Also, from choosing $\prod\limits_{j=1}^{t}{ \bigg(\dfrac{p_j}{p_j-1}}\bigg) \gg 2\varphi(m)$ and the reasoning above, we conclude that for each Chebyshev polynomial of the first kind $T(x)$ and $\varepsilon > 0$, there exist infinitely many positive integers $n$ such that $P^+(T(n)) < n^{\varepsilon}.$ By a similar approach, we find that the product of Chebyshev polynomials also belongs to $\mathcal{P},$ as proven in the theorem below.
\begin{theorem}
Given $k$ Chebyshev polynomials of the first kind, denoted as $T_{m_1}(x),$ 

$T_{m_2}(x),\ldots,T_{m_k}(x)$. Then, $T(x) = T_{m_1}(x)T_{m_2}(x)\cdots T_{m_k}(x) \in \mathcal{P}.$
\end{theorem}

\begin{proof}
We set $l = max \{m_1,m_2,\ldots, m_k\}$ and choose the first $t$ prime numbers $p_1,p_2,\ldots, p_t$ that sastify $$\prod_{j=1}^{t}{ \bigg(\dfrac{p_j}{p_j-1}}\bigg) \gg 2\varphi(l).$$
We take $n = T_{p_1p_2\cdots p_t}(m)$ with $m$ being a positive integer. Then we have:
\begin{align*}
T(n) &= \prod\limits_{1\le i \le k} T_{m_i}(T_{p_1p_2\cdots p_t}(m)) \\
& = \prod\limits_{1 \le i \le k } T_{m_i \cdot p_1p_2\cdots p_t}(m) \\ 
& = \dfrac{1}{2^k} \prod\limits_{1\le i \le k} \prod_{\substack{d\mid P_i \\ P_i/d : \text{odd}}} \psi_{4d}(m) 
\end{align*}
where $P_i = m_i \cdot p_1p_2\cdots p_t$
\newline 
We consider for each $d \mid P_i$ for all $i = \overline{1,k}$, we have  
\begin{align*}
\text{deg}\left(\psi_{4d}(2x)\right)=\dfrac{\varphi(4d)}{2}&\le 2\varphi(m_i)\cdot (p_1-1)(p_2-1)\dots (p_t-1)\\
&=2\varphi(m_i) \cdot \left(\prod_{j=1}^{t}\left(\dfrac{p_j-1}{p_j}\right)\right)p_1p_2 \cdots p_t \ll p_1p_2 \cdots p_t.
\end{align*}
Therefore, we see that $T(x)$ is the product of polynomials with degrees less than the degree of the polynomial $T_{p_1p_2\cdots p_t}(x).$ By applying Lemma \ref{lemma 3.1}, we achieve the desired result for the problem.
\end{proof}
To conclude, based on \cite{masakazu}, we have the factorization of  Chebyshev polynomials of types 2, 3, and 4 is similar to that of type 1. Thus, by using a similar approach, we also have $T(x) \in \mathcal{P}$ and $T(x)$ satisfy $P^+(T(n)) < n^{\varepsilon}$ for infinitely many positive integers $n$, where $T(x)$ represents the three aforementioned types of polynomials.


\begin{thebibliography}{99}
\bibitem{balog}
A. Balog and T. D. Wooley, On Strings of Consecutive Integers with No Large Prime Factors, \textit{J. Austral. Math. Soc. Ser. A Pure Math. Statist.}, {\bf 64}(2) (1998) 266–276.

\bibitem{bober}
J. W. Bober et al., Smooth Values of Polynomials, \textit{J. Austral. Math. Soc.}, {\bf 108}(2) (2020) 245–261.

\bibitem{bunyakovsky}
V. Bouiakowsky, On the invariable numerical divisors of entire rational functions, \textit{Mém. Acad. Sc. St. Petersbourg}, 6th series, vol. VI, 1857, pp. 305–329.

\bibitem{euler}
L. Euler, De progressionibus harmonicis observationes, \textit{Euler Archive} (1740).

\bibitem{hooley}
C. Hooley, On the greatest prime factor of a cubic polynomial, \textit{J. für die reine und angewandte Mathematik}, {\bf 303-304} (1978) 21–50.

\bibitem{jori}
J. Merikoski, On the largest prime factor of $n^2 + 1$, \textit{J. Eur. Math. Soc.}, {\bf 25}(4) 2023 (1253–1284).

\bibitem{keates}
M. Keates, On the greatest prime factor of a polynomial, \textit{Proc. Edinburgh Math. Soc.}, {\bf 16} (1969) 301–303.

\bibitem{lang}
S. Lang, \textit{Algebra}, Revised 3rd edition, Graduate Texts in Mathematics, Vol. 211 (Springer-Verlag, New York, 2002).

\bibitem{legendre} A. M. Legendre, \textit{Théorie des Nombres}, Firmin Didot Frères, Paris, (1830). 

\bibitem{masakazu}
M. Yamagishi, A note on Chebyshev polynomials, cyclotomic polynomials, and twin primes, \textit{J. Number Theory}, {\bf 133}(7) 2013 2455–2463.

\bibitem{mertens}
F. Mertens, Ein Beitrag zur analytischen Zahlentheorie, \textit{J. für die reine und angewandte Mathematik}, {\bf 78} (1874) 46–62.

\bibitem{murty}
M. Ram Murty, Polynomials assuming square values, 2012.

\bibitem{rayes}
M. O. Rayes, V. Trevisan, and P. S. Wang, Factorization properties of Chebyshev polynomials, \textit{Comput Math Appl}, {\bf 50}(8) 2005 1231–1240.

\bibitem{schinzel}
A. Schinzel, On two theorems of Gel’fond and some of their applications, \textit{Acta Arith}, {\bf 13}(2) 1967 177–236.

\bibitem{schur}
I. Schur, Über die Existenz unendlich vieler Primzahlen in einiger speziellen arithmetischen Progressionen, S.-B. Berlin, Math. Ges., {\bf 11} (1912) 40–50.

\bibitem{tattershall}
J. Tattarsall, \textit{Representations}, Elementary Number Theory in Nine Chapters, Cambridge: Cambridge University Press, 2005 258–303.

\bibitem{tenenbaum}
G. Tenenbaum, Sur une question d’Erdős et Schinzel, II, \textit{Inventiones mathematicae}, {\bf 99}(1) (1990) 215–224.
\end{thebibliography}
\end{document}